\documentclass[twoside]{amsart}
\usepackage{amssymb,amsmath,amsopn}
\usepackage[dvips]{graphicx}   
\usepackage{color,epsfig}      
\usepackage{latexsym}

\usepackage{amsmath,amssymb,enumerate}
\usepackage[margin=2.5cm,nohead]{geometry}
\usepackage[utf8]{inputenc}
\usepackage[english]{babel}
\usepackage{latexsym}
\usepackage{amsfonts}
\usepackage{amssymb}
\usepackage{amsthm}
\usepackage{setspace}

\usepackage{lmodern}
\usepackage{indentfirst}
\usepackage{graphicx}
\usepackage{geometry}

\newtheorem{thm}{Theorem}
\newtheorem{prop}{Proposition}
\newtheorem{lem}[prop]{Lemma}
\theoremstyle{definition}
 
\newtheorem{rem}[prop]{Remark}

\newcommand{\norm}[1]{\left\lVert#1\right\rVert}
\newcommand{\vol}[1]{\mathrm{Vol}\left(#1\right)}

\renewcommand{\Re}{\mathbb R}
\newcommand{\Red}{\Re^d}
\DeclareMathOperator{\intof}{int}
\DeclareMathOperator{\conv}{conv}
\DeclareMathOperator{\bd}{bd}

\begin{document}

\title[] {Bounds on Convex Bodies in Pairwise Intersecting Minkowski Arrangement of Order $\mu$}

\author[Vikt\'oria F\"oldv\'ari]{Vikt\'oria F\"oldv\'ari}     

\address{Vikt\'oria F\"oldv\'ari,
Institute of Mathematics, E\"otv\"os Lor\'and University, Budapest, Hungary,
}
\email{foldvari@math.elte.hu}
                                 

                                                                                
\begin{abstract}
The $\mu$-kernel of an $o$-symmetric convex body is obtained by shrinking the body about its center by a factor of $\mu$.
  As a generalization of pairwise intersecting Minkowski arrangement of $o$-symmetric convex bodies, we can define the pairwise intersecting Minkowski arrangement of order $\mu$. Here, the homothetic copies of an $o$-symmetric convex body are so that none of their interiors intersect the $\mu$-kernel of any other.
  We give general upper and lower bounds on the cardinality of such arrangements, and study two special cases:
  For $d$-dimensional translates in classical pairwise intersecting Minkowski arrangement we prove that the sharp upper bound
  is $3^d$.
  The case $\mu=1$ is the Bezdek--Pach Conjecture, which asserts that the maximum number of pairwise 
  touching positive homothetic copies of a convex body in $\Re^d$ is $2^d$. We 
  verify the conjecture on the plane, that is, when $d=2$. Indeed, we show that 
  the number in question is four for any planar convex body.
\end{abstract}
\maketitle

\section{Introduction}

A \emph{positive homothetic copy} of a convex body (i.e. a compact convex set 
with non-empty interior) $K$ in Euclidean $d$-space $\Red$ is a set of the form 
$\lambda K+t$ where $\lambda >0$ and $t\in\Red$. 
Two sets in $\mathbb{R}^d$ are said to \emph{touch} each other if they intersect but 
their interiors are disjoint.

The following notion was introduced by Fejes Tóth \cite{FT65}: Pairwise intersecting homothets of a centrally symmetric convex body in the $d$-dimensional Euclidean
space form a \emph{Minkowski arrangement} if none of them contains the center of any other in its interior.
In this paper, we only consider Minkowski arrangements that are pairwise intersecting.

Polyanskii \cite{P17} recently proved that such a family of convex bodies has at most $3^{d+1}$
members. This result was improved by Naszódi and Swanepoel \cite{NS17} showing an upper bound of $2\cdot 3^d$.
It is natural to conjecture that the maximum number of elements is $3^d$.

We prove the following upper bound on the cardinality of
a family containing translates of a centrally symmetric convex
body in pairwise intersecting Minkowski arrangement in Section \ref{sec:Minkowski}:

\begin{thm}\label{thm:Minkowski3d}
In $\mathbb{R}^d$, a pairwise intersecting Minkowski arrangement consisting of translates of a centrally symmetric convex body $K$ contains at most $3^d$ elements. This bound is sharp, equality holds if and only if $K$ is a $d$-dimensional parallelotope.
\end{thm}

We show a construction for arbitrary centrally symmetric convex body that gives a linear lower bound on the cardinality of maximal pairwise intersecting Minkowski arrangements of translates. Although from Theorem 3 of \cite{NPS17} we can deduce the existence of an exponential lower bound, we now give a simple and deterministic construction.

\begin{prop}\label{thm:lowerbound}
For a centrally symmetric convex body $K$ in $\mathbb{R}^d$ ($d\geq 2$), a maximum cardinality set consisting of translates of $K$ in pairwise intersecting Minkowski arrangement has at least $2d+3$ elements.
\end{prop}

We introduce some generalizations of the problem based on an idea of Böröczky and Szabó \cite{BSZ04}:
For $0\leq \mu \leq 1$ they defined the \emph{$\mu$-kernel} of an $o$-symmetric convex body $K$ as $\mu K$.




Using this notion, for homothets of an $o$-symmetric convex body we can consider a
\emph{pairwise intersecting Minkowski arrangement of order $\mu$}, where the homothets are pairwise intersecting but none of their interiors intersect
the center of any other. 

We prove an upper bound on the cardinality of such an arrangement, then, for centrally symmetric convex bodies we verify the existence of an exponential lower bound.

\begin{thm}\label{thm:muMink}
In $\mathbb{R}^d$, a pairwise intersecting Minkowski arrangement of order $\mu$ consisting of translates of a centrally symmetric convex body $K$ contains at most $ \left( 1 + \frac{2}{1+\mu}  \right)^d $ elements.
\end{thm}

\begin{prop}\label{thm:exponential}
Let $M_\mu (K)$ denote the maximum number of translates of a $d$-dimensional, $o$-symmetric convex body $K$ in pairwise intersecting $\mu$-Minkowski arrangement. For $\mu<\sqrt{2}-1$, there exists a lower bound $M_\mu (K)\geq e^{cd}$ for some universal constant $c$.
\end{prop}


In 1962, Danzer and Gr\"unbaum \cite{DG62} proved that the maximum cardinality 
of a family 
of pairwise touching translates of a convex body $K$ in $\Red$ is $2^d$, which 
bound is attained if and only if $K$ is an affine image of a cube. 
Petty \cite{P71} showed that every convex body in the plane (or in 3-space) 
has three (four) pairwise touching translates.
As an extension of this problem, Bezdek and Pach \cite{BC88} conjectured in 
1988 that the maximum number of pairwise touching positive homothetic copies of 
a convex body in $\Re^d$ is $2^d$. They showed that any such family of 
homothetic copies has at most $3^d$ elements, and if $K$ is a $d$-dimensional 
Euclidean ball, then the maximum is equal to $d+2$. Naszódi
\cite{N06} improved the first estimate by proving the upper 
bound $2^{d+1}$. In \cite{LN09}, L\'angi and Naszódi proved 
(using a result \cite{BB03} of Bezdek and Brass about one-sided Hadwiger 
numbers) the upper bound $3\cdot2^{d-1}$ in the case when $K$ is centrally 
symmetric.

In Section \ref{sec:BezdekPach}, we show that the conjecture holds on the plane, moreover, 
every planar convex body has four pairwise touching homothets.
\begin{thm}\label{thm:BPplane}
For any convex body $K$ in $\Re^2$, the maximum number of pairwise touching 
positive homothetic copies of $K$ is four.
\end{thm}

The generalized notion of Minkowski arrangement provides a connection between the original problem of pairwise intersecting
Minkowski arrangements and the Bezdek--Pach Conjecture \cite{BC88}. In both problems we consider
pairwise intersecting Minkowski arrangements of order $\mu$, in the first case $\mu=0$, while in the latter case $\mu=1$.

For two points $a,b$ in $\Re^d$, we denote the closed and the open line 
segment connecting them by $[a,b]$ and $(a,b)$, respectively. We use the standard 
notations $\conv$, $\bd$ and $\intof$ to denote the convex hull, the boundary and the interior of a set 
in $\Re^d$, respectively.

In Section \ref{sec:Minkowski}, we prove Theorems \ref{thm:Minkowski3d}, \ref{thm:muMink}, and Propositions \ref{thm:lowerbound} and \ref{thm:exponential}. Sections \ref{sec:BezdekPach} and \ref{sec:lowerbound} together give the proof of Theorem \ref{thm:BPplane}. Finally, in Section \ref{sec:topologicalnote}, we verify Proposition \ref{prop:elhelyezkedesTop}, a topological note that yields to an alternative version of the proof of Theorem \ref{thm:BPplane}.

$\mathbf{Acknowledgement:}$ This work was made under the supervision of Márton Naszódi, who called my attention to this topic. I would like to express my gratitude for his aid, the corrections and all the useful discussions. I thank Géza Tóth the idea of the much shorter proof of Proposition \ref{prop:k5sikbarajzolhato}. This research was supported by the ÚNKP-17-3 New National Excellence Program of the Ministry of Human Capacities.

\section{Bounds on pairwise intersecting Minkowski arrangements}\label{sec:Minkowski}


It is natural to conjecture that in $\mathbb{R}^d$, a pairwise intersecting Minkowski arrangement consisting of homothets of a centrally symmetric convex body contains at most $3^d$ elements. Here we prove this upper bound -- and a generalization -- for the case when all the homothets in the arrangement are translates of the given body.

\subsection{Proof of Theorem \ref{thm:Minkowski3d} and \ref{thm:muMink}}
$\ $

First, we verify Theorem \ref{thm:muMink}, then Theorem \ref{thm:Minkowski3d} will follow as a corollary.

Any $o$-symmetric convex body $K$ can be considered as the unit ball of a normed space $(\mathbb{R}^d,\norm{. }_K)$, where for any $x\in \mathbb{R}^d$, $\norm{x}_K= \inf \{\lambda\in \mathbb{R^+} | x\in \lambda K\}$.
It is easy to see that having a pairwise intersecting Minkowski arrangement of order $\mu$ is equivalent to the following two conditions on the distances between centers: none of them can be farther than 2, nor closer than $1+\mu$ to any other. After applying a homothety, this is equivalent to the problem when the distances are between 1 and $\frac{2}{1+\mu}$.

\begin{lem}\label{lem:terfogatlemma}
Consider a centrally symmetric convex body $K$ in $\mathbb{R}^d$ and $v_1,v_2,...,v_n \in \mathbb{R}^d$, so that $1 \leq \norm{v_i-v_j}_K \leq \lambda $ for any $i\neq j$. Then $n\leq (\lambda+1)^d$.
\end{lem}

\begin{proof}
By the assumption, for different indices the bodies $v_i+\frac{1}{2}K$ are pairwise non-overlapping.
Let $Q=\conv \left[  \bigcup\limits_{i=1}^n \left( v_i+\frac{1}{2}K\right)\right]$. Since $\mathrm{diam}_K (Q) \leq \lambda+1$, using the isodiametric inequality for Minkowski spaces \cite{FLM90} we get that

\begin{equation}\label{terfogatlemma}
\frac{n}{2^d}\vol{K} \leq \vol{Q} \leq \vol{\frac{\lambda+1}{2}K}.
\end{equation}
From this, $n \leq (\lambda+1)^d$ follows.
\end{proof}

Applying this lemma for $\lambda=\frac{2}{1+\mu}$, we get that the number of points with pairwise distances between $1$ and $\frac{2}{1+\mu}$ is at most $ \left( 1 + \frac{2}{1+\mu}  \right)^d $, which is equivalent to the statement of Theorem \ref{thm:muMink}.
\qed


Theorem \ref{thm:Minkowski3d} is the special case of Theorem \ref{thm:muMink} with $\mu=0$, so the upper bound $3^d$ follows easily.

To reach this, \eqref{terfogatlemma} has to hold with two equalities. From the following lemma of Groemer \cite{G61} we can see that this happens if and only if $K$ is a $d$-dimensional parallelotope.

\begin{lem}\label{Groemer}
Suppose that $K$ is a convex body in $\mathbb{R}^d$ such that for some $1<t\in \mathbb{R}$ the body $tK$ can be decomposed into translates of $K$. Then $K$ is a $d$-dimensional parallelotope and $t$ is an integer. The partition is unique.
\end{lem}
\qed

\begin{rem}
The bound in Theorem \ref{thm:muMink} gives the known result $2^d$ for the pairwise touching case, when $\mu=1$.
\end{rem}


\subsection{Proof of Proposition \ref{thm:lowerbound}}
$\ $

First, we show a construction of seven bodies in $\mathbb{R}^2$, then the lower bound $2d+3$ for the higher dimensional cases will follow recursively.
In $\mathbb{R}^2$, consider an affine-regular hexagon inscribed in $K$ that is symmetric about the center of $K$ (see for example \cite[Lemma 4.3]{PA91}). There exist seven translates of this hexagon in Minkowski arrangement, shown in Figure \ref{fig:hatszogek}. Translate $K$ in a way that the center points are the same as the centers of the above hexagons. Now a center of any translate is either not contained in another body, or lies on its boundary. Furthermore, these translates share a common point, so they are pairwise intersecting. This means, that the construction gives a Minkowski arrangement.

\begin{figure}[h]
\centering
\includegraphics[width=0.4\textwidth]{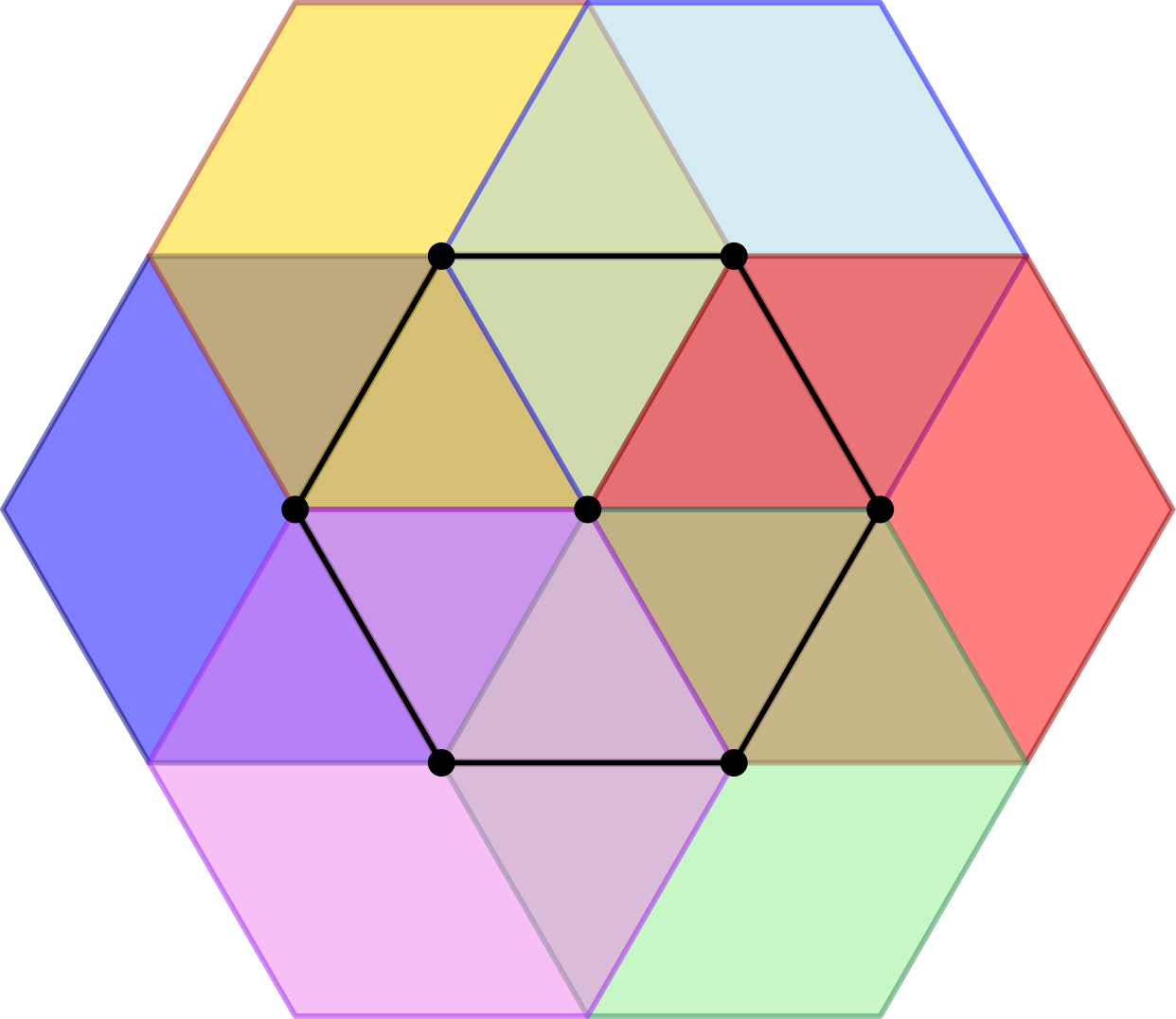}
\caption{Seven translates of an affine-regular hexagon in Minkowski arrangement}
\label{fig:hatszogek}
\end{figure}

For a centrally symmetric convex body $K$ in $\mathbb{R}^d$, denote by $M(K)$ the maximal number of translates in a  pairwise intersecting Minkowski arrangement. It is easy to see, that for any $K$ in $\mathbb{R}^1$, $M(K)=3$, and we showed that for $K$ in $\mathbb{R}^2$, $M(K)\geq 7$.

Let $e_1,...,e_d$ be an Auerbach basis \cite[Chapter 3]{thompson_1996} of the space $(\mathbb{R}^d, \norm{ .}_K)$. In dimension $d\geq 3$, using the above planar construction, we can take 7 translates of $K$ in a Minkowski arrangement such that their centers lie in the plane of the first two basis vectors $e_1$ and $e_2$. Along each direction $e_3,...,e_d$ we can add two further translates of $K$ to the arrangement so that they contain $o$ on their boundary.
\qed


Now we verify Proposition \ref{thm:exponential}. Note that Theorem 3. of \cite{NPS17} by Naszódi, Pach and Swanepoel implies the existence of an exponential lower bound for translates in pairwise intersecting Minkowski arrangement. Their idea was based on a result of Arias-de-Reyna, Ball, and Villa \cite{arias_de_reyna}. Here we give a similar argument for $\mu$-Minkowski arrangements.

\begin{proof}[Proof of Proposition \ref{thm:exponential}]
The statement follows from a result of Bourgain \cite{B86}. He showed that on the unit sphere of any normed space, there is an exponentially large number of points so that the distance of any two is more than $\sqrt{2}-\varepsilon$. Consider the $o$-symmetric convex body $K$ as the unit ball of the normed space $(\mathbb{R}^d,\norm{. }_K)$. Choosing $\mu<\sqrt{2}-1$, we get exponentially many points on the sphere so that their pairwise distances are between $1+\mu$ and 2. Considering these points as centers, we verify the statement.
\end{proof}

\section{Proof of the upper bound in Theorem~\ref{thm:BPplane}}\label{sec:BezdekPach}


Let $\mathcal{K}=\{K_1,K_2,\ldots,K_n\}$ be a family of pairwise touching positive 
homothetic copies of a planar convex body $K$. 

If there is a point that belongs to four of the homothets, then we can 
enlarge (or shrink) each of the four bodies from that point as a center, to 
obtain four 
touching translates of $K$. By the result of Danzer and Gr\"unbaum \cite{DG62},
 this implies that $K$ is a parallelogram. It is easy to see that in this case, the 
family does not have a fifth member. Thus, from this point on, we will assume 
that no point belongs to four of the homothets.

If there is a point that belongs to three of the homothets and $\mathcal{K}$ has at least 
four members, then we will show that this point also belongs to a fourth body.

\begin{prop}\label{nincs_haromszoros}
Let $K_1,\  K_2,\ K_3,\ K_4$ be pairwise touching positive homothets 
of the convex body $K$ in $\mathbb{R}^2$. If $K_1\cap K_2\cap K_3\neq \emptyset$,
 then $K_1\cap K_2\cap K_3\cap K_4\neq \emptyset$.
\end{prop}
\begin{proof}
Let  $p\in K_1\cap K_2\cap K_3$, and $C_i$ be the smallest angular region with vertex $p$ containing $K_i$.

We show that $\intof C_i \cap \intof C_j=\emptyset$
for any $i\neq j$, $i,j\in\{1,2,3\}$.

Suppose that for a pair $i\neq j$ there exists $c\in(\intof C_i \cap \intof C_j)$. 
Then the line $pc$ intersects the interior of both $K_i$ and $K_j$ because 
$C_i$ and $C_j$ are the smallest angular regions containing $K_i$ and $K_j$ respectively.
 Hence due to the convexity of the bodies, $K_i$ overlaps $K_j$, which is a contradiction.

Suppose that $K_1\cap K_2\cap K_3\cap K_4= \emptyset$. Then $p\notin K_4$,
thus there exists a supporting line $\ell$ of $K_4$ that does not go through $p$
and separates $K_4$ from $p$. $K_4$ touches $K_1$, $K_2$ and $K_3$, hence each of 
these three bodies has a point in both of the closed half-planes bounded by $\ell$.
From this it follows that $\ell$ intersects the angular regions $C_1$, $C_2$ and $C_3$. For every 
$i\in\{1,2,3\}$, $\ell\cap C_i$ is a connected subset of $\ell$, thus there is a middle
one of them. Without loss of generality we can assume that this one is $K_1$. Let $v_1=p-x_1$. The image of $p$ by the homothety that maps
$K_1$ to $K_4$ is the point $x_4+\frac{\lambda_4}{\lambda_1}\cdot v_1 \in \ell$. The same 
homothety maps $C_1$ to the angular region $C_1^{\prime}:=C_1+\left(x_4-x_1+v_1\cdot \left(\frac{\lambda_4}{\lambda_1}-1\right)\right)$. 
(Figure \ref{fig:nincsharomszoros}) As $K_1\subset C_1$ and the bodies are positive
homothets, $K_4\subset C_1^{\prime}$ follows. At least one pair of the bounding
lines of $C_1$ and $C_1^{\prime}$ are different, thus due to the fact that
$\intof C_i \cap \intof C_j=\emptyset$ for any $i\neq j$, $i,j\in\{1,2,3\}$
 $C_1^{\prime}$ is disjoint to at least one of the angular regions $C_2$ and $C_3$. But in this case
$K_4$ cannot touch the body lying in this angular region, which is a contradiction.

\end{proof}

Thus, it is enough to consider the case when no point belongs to three of the homothets.

\begin{prop}\label{prop:k5sikbarajzolhato}
Let $K_1$, $K_2$,...,$K_n$ be pairwise touching convex bodies in $\Re^2$, 
such that no three share a common point. Then $n\leq 4$.
\end{prop}
\begin{proof}
For each $i\in \{1,...,n\}$, choose an interior point $p_i\in K_i$. The bodies are pairwise touching, so we can draw a curve between any two of the chosen points $p_i$, $p_j$ so that it lies in $K_i \cup K_j$. Since no three of the bodies share a common point, these curves intersect only in the interior of the bodies. It is easy to see that we can eliminate these intersections with a perturbation. This way we draw the complete graph of $n$ vertices on the plane, from which $n \leq 4$ follows immediately. 
\end{proof}

\begin{figure}[h]
\centering
\includegraphics[width=0.55\textwidth]{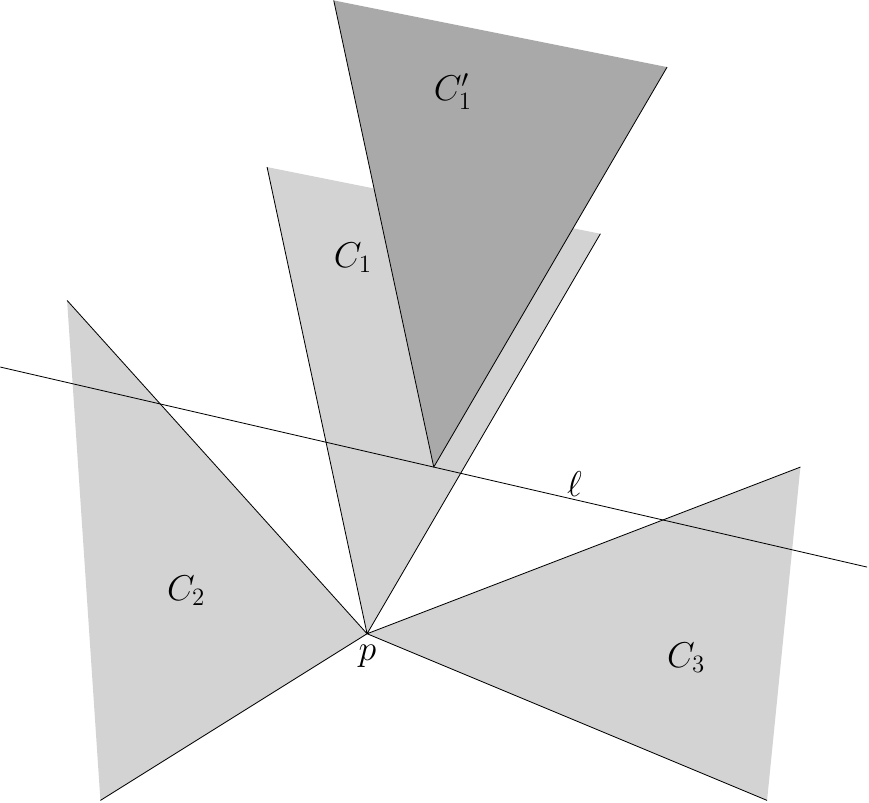}
\caption{The smallest angular regions containing the bodies}
\label{fig:nincsharomszoros}
\end{figure}
%
%
%
%

\section{Proof of the lower bound in Theorem~\ref{thm:BPplane}}\label{sec:lowerbound}

In this section, we show that for any planar convex body $K$, there are four pairwise touching 
homothets of $K$.

Consider two distinct parallel support lines of $K$ that each touch $K$ at one 
point: $x_1$ and $x_2$. The existence of such pair of lines follows from Theorem 2.2.9. of \cite{Sch14}, but may also 
be proved as an exercise.

Let $K_1=K$ and $K_2=K+x_2-x_1$. Let $f$ be the line through the single point 
of contact, $x_2$, of $K_1$ and $K_2$ parallel to $x_2-x_1$. On both 
sides of $f$, there is a translate of $K$ that touches both $K_1$ and $K_2$. 
Indeed, if we push $K$ around $K_1$ so that it always touches $K_1$ then, by 
continuity, such two positions will be found.

If on both sides we can find such translates of $K$ that also contains $x_2$ 
then $x_2$ is a common point of four translates of $K$ and we are done. Thus we 
assume that at least one of these translates does not contain $x_2$. We call 
this translate $K_3$. 

Now, $K_1,K_2,K_3$ are pairwise touching translates of $K$ that do not share a 
common point. It follows that they surround a bounded region $R$ with non-empty 
interior. Consider the largest homothet $K_4$ of $K$ contained in $R$. To 
finish the proof, we claim that $K_4$ touches $K_1,K_2$ and $K_3$. Indeed, 
assume $K_4$ touches only two of them, say $K_1$ and $K_2$. Consider a line 
that 
separates $K_4$ and $K_1$, and another line that separates $K_4$ and $K_2$. Let 
$u_1$ and $u_2$ be the unit normal vectors of these two lines respectively, 
pointing away from $K_4$. Clearly, if the origin is not in $\conv(u_1,u_2)$ 
then 
$K_4$ can be moved a little inside $R$ so that it does not touch either $K_1, 
K_2$ or $K_3$. Then, we may enlarge $K_4$ slightly within $R$ contradicting the 
maximality of $K_4$. Thus $o\in\conv(u_1,u_2)$, that is $u_1=-u_2$. However, in 
this case, $K_1$ and $K_2$ are strictly separated, which is a contradiction, finishing the proof of the lower bound in Theorem \ref{thm:BPplane}.

\section{A topological note}\label{sec:topologicalnote}

In this section we present Proposition \ref{prop:elhelyezkedesTop}, a topological observation, which may be used in place of Proposition \ref{prop:k5sikbarajzolhato} to prove the upper bound in Theorem \ref{thm:BPplane}.


An \emph{arc} in the plane is the image of an injective continuous map 
of the $[0,1]$ interval into the plane.
A \emph{Jordan curve} in the plane is the image of an injective continuous map 
of the circle into the plane.
We will call the closed bounded region bounded by a Jordan curve a \emph{Jordan 
region}.

Let $K_1,K_2,K_3$ be three pairwise touching Jordan regions whose pairwise 
intersections are non-empty arcs (which may be degenerate, that is a single 
point). Using the Jordan curve theorem, it is easy to show that the complement 
of 
$K_1\cup K_2\cup K_3$ in the plane has two connected components, one 
bounded and one unbounded. We call the closure of the bounded component the 
\emph{internal region} surrounded by $K_1,K_2, K_3$, and the closure of the 
unbounded component the \emph{external region}.

\begin{prop}\label{prop:elhelyezkedesTop}
Let $K_1,K_2,K_3,K_4$ be four pairwise touching Jordan regions whose pairwise 
intersections are non-empty arcs (which may be degenerate, that is a single 
point). Suppose that $K_1\cap K_2 \cap K_3 \cap K_4=\emptyset$. Then one of them lies in the internal region surrounded by the other 
three.
\end{prop}

\begin{proof}
 
We will call the image of the non-negative reals under an injective mapping 
into the plane an \emph{unbounded path} if it is an unbounded subset of the 
plane. The image of $0$ is the \emph{starting point} of the unbounded path.

Assume that $K_1$ is not in the internal region 
surrounded by $K_2,K_3,K_4$. Then there is a point $p_1$ on the boundary of 
$K_1$ that does not belong to either of the other three sets, and from which 
there is an unbounded path, $\gamma_1$ disjoint from the other three 
sets.
Similarly, if $K_2$ is not in the internal region surrounded by the other 
three, then there is a point $p_2$ on the boundary of 
$K_2$ that does not belong to either of the other three sets, and from which 
there is an unbounded path, $\gamma_2$ disjoint from the other three 
sets. And the same holds for $K_3$ yielding $p_3$ and $\gamma_3$.

We may assume that $\gamma_1,\gamma_2$ and $\gamma_3$ are pairwise disjoint. 
Now, $\gamma_1\cup \gamma_2 \cup \gamma_3$ partition the external 
region of $K_1,K_2,K_3$ into three parts. And $K_4$ is in one of these three 
parts. However, each part only intersects two of the sets $K_1,K_2,K_3$, which
 is a contradiction.
\end{proof}

\begin{rem}
Although Proposition \ref{prop:elhelyezkedesTop} has been observed in \cite{Bezdek1995}, their proof does not work for the special case when three of the bodies share a common point. Therefore we found it necessary to give a general proof.
\end{rem}

Observe that the conclusion of Proposition \ref{prop:k5sikbarajzolhato} follows from Proposition~\ref{prop:elhelyezkedesTop}.
Indeed, we may assume that $K_4$ is in the internal region $I$ surrounded by 
$K_1,K_2,K_3$.
Suppose that $n\geq 5$. Since $K_5$ touches $K_4$, it must also lie in $I$.
On the other hand, $K_5$ touches $K_1,K_2,K_3$ at points that do not belong to 
$K_4$. Now, $( \bd I)\setminus K_4$ is the union of three open arcs, and $K_4$ 
must have a point on at least two of these arcs to touch $K_1,K_2,K_3$. 
However, then the interior of $K_5$ intersects the interior of at least one set 
from $K_1,\ldots,K_4$, a contradiction.

\bibliographystyle{amsplain}
\bibliography{biblio}

\providecommand{\bysame}{\leavevmode\hbox to3em{\hrulefill}\thinspace}
\providecommand{\MR}{\relax\ifhmode\unskip\space\fi MR }
\providecommand{\MRhref}[2]{%
  \href{http://www.ams.org/mathscinet-getitem?mr=#1}{#2}
}
\providecommand{\href}[2]{#2}
\begin{thebibliography}{10}

\bibitem{arias_de_reyna}
J.~Arias-de Reyna, K.~Ball, and R.~Villa, \emph{Concentration of the distance
  in finite dimensional normed spaces}, Mathematika \textbf{45} (1998), 245 --
  252.

\bibitem{Bezdek1995}
A.~Bezdek, K.~Kuperberg, and W.~Kuperberg, \emph{Mutually contiguous translates
  of a plane disk}, Duke Mathematical Journal \textbf{78} (1995), no.~1,
  19--31.

\bibitem{BB03}
K.~Bezdek and P.~Brass, \emph{On {$k\sp +$}-neighbour packings and one-sided
  {H}adwiger configurations}, Beitr\"age Algebra Geom. \textbf{44} (2003),
  no.~2, 493--498.

\bibitem{BC88}
K.~Bezdek and R.~Connelly, \emph{Intersection points}, Ann. Univ. Sci.
  Budapest. E\"otv\"os Sect. Math. \textbf{31} (1988), 115--127 (1989).

\bibitem{B86}
J.~Bourgain, V.~Milman, and H.~Wolfson, \emph{On type of metric spaces},
  Transactions of the American Mathematical Society \textbf{294} (1986), no.~1,
  295--295.

\bibitem{BSZ04}
K.~Böröczky and L.~Szabó, \emph{Minkowski arrangements of spheres},
  Monatshefte für Mathematik \textbf{141} (2004), no.~1, 11--19.

\bibitem{DG62}
L.~Danzer and B.~Gr{\"u}nbaum, \emph{\"{U}ber zwei {P}robleme bez\"uglich
  konvexer {K}\"orper von {P}. {E}rd{\H o}s und von {V}. {L}. {K}lee}, Math. Z.
  \textbf{79} (1962), 95--99.

\bibitem{FLM90}
Z.~F{\"u}redi, J.~C. Lagarias, and F.~Morgan, \emph{Singularities of minimal
  surfaces and networks and related extremal problems in minkowski space},
  Discrete and Computational Geometry, 1990.

\bibitem{G61}
H.~Groemer, \emph{Abschätzungen für die anzahl der konvexen körper, die
  einen konvexen körper berühren.}, Monatshefte für Mathematik \textbf{65}
  (1961), 74--81.

\bibitem{LN09}
Zs. L{\'a}ngi and M.~Nasz{\'o}di, \emph{On the {B}ezdek-{P}ach conjecture for
  centrally symmetric convex bodies}, Canad. Math. Bull. \textbf{52} (2009),
  no.~3, 407--415.

\bibitem{NPS17}
M.~Naszódi, J.~Pach, and K.~Swanepoel, \emph{Arrangements of homothets of a
  convex body}, Mathematika \textbf{63} (2017), no.~2, 696–710.

\bibitem{NS17}
M.~Naszódi and K.~Swanepoel, \emph{Arrangements of homothets of a convex body
  {II}}, Contributions to Discrete Mathematics \textbf{13} (2017).

\bibitem{N06}
M.~Nasz{\'o}di, \emph{On a conjecture of {K}\'aroly {B}ezdek and {J}\'anos
  {P}ach}, Period. Math. Hungar. \textbf{53} (2006), no.~1-2, 227--230.

\bibitem{PA91}
J.~Pach and K.~P. Agarwal, \emph{Combinatorial geometry}, Tech. report, Durham,
  NC, USA, 1991.

\bibitem{P71}
C.~M. Petty, \emph{Equilateral sets in {M}inkowski spaces}, Proc. Amer. Math.
  Soc. \textbf{29} (1971), 369--374.

\bibitem{P17}
A.~Polyanskii, \emph{Pairwise intersecting homothets of a convex body},
  Discrete Mathematics \textbf{340} (2017), no.~8, 1950--1956.

\bibitem{Sch14}
R.~Schneider, \emph{Convex bodies: the {B}runn-{M}inkowski theory}, expanded
  ed., Encyclopedia of Mathematics and its Applications, vol. 151, Cambridge
  University Press, Cambridge, 2014.

\bibitem{thompson_1996}
A.~C. Thompson, \emph{Minkowski geometry}, Encyclopedia of Mathematics and its
  Applications, Cambridge University Press, 1996.

\bibitem{FT65}
L.~Fejes Tóth, \emph{Minkowskian distribution of discs}, Proceedings of the
  American Mathematical Society \textbf{16} (1965), no.~5, 999--1004.

\end{thebibliography}
\end{document}